\documentclass[11pt]{article}

\usepackage[margin=2.25cm]{geometry}

\usepackage{amsmath}
\usepackage{amssymb}
\usepackage{amsthm}
\usepackage{amsfonts}
\usepackage{enumerate}
\usepackage{graphicx}
\usepackage{url}

\usepackage{ytableau}

\usepackage{tikz}
\usetikzlibrary{arrows,shapes,positioning,decorations.markings,decorations.pathreplacing}

\newtheorem{theorem}{Theorem}[section]

\newtheorem{lemma}[theorem]{Lemma}
\newtheorem{example}[theorem]{Example}
\newtheorem{corollary}[theorem]{Corollary}
\newtheorem{proposition}[theorem]{Proposition}

\def\R{\mathbb{R}}
\def\Z{\mathbb{Z}}

\DeclareMathOperator{\inv}{n_{inv}}
\DeclareMathOperator{\st}{st}
\DeclareMathOperator{\hgt}{ht}
\DeclareMathOperator{\dep}{dp}
\DeclareMathOperator{\depm}{dp^{\kern-1pt *} \kern-1pt}
\DeclareMathOperator{\ret}{ret}
\DeclareMathOperator{\Ret}{\overrightarrow{\text{Ret}}}

\begin{document}

\title{\bf Generating Functions for Inverted Semistandard Young Tableaux \& Generalized Ballot Numbers}
\author{Paul Drube\\
\small Department of Mathematics and Statistics\\[-0.8ex]
\small Valparaiso University\\[-0.8ex] 
\small Valparaiso, Indiana, U.S.A.\\
\small\tt paul.drube@valpo.edu\\
}

\maketitle

\begin{abstract}
An inverted semistandard Young tableau is a row-standard tableau along with a collection of inversion pairs that quantify how far the tableau is from being column semistandard.  Such a tableau with precisely $k$ inversion pairs is said to be a $k$-inverted semistandard Young tableau.  Building upon earlier work by Fresse and the author, this paper develops generating functions for the numbers of $k$-inverted semistandard Young tableau of various shapes $\lambda$ and contents $\mu$.  An easily-calculable generating function is given for the number of $k$-inverted semistandard Young tableau that ``standardize" to a fixed semistandard Young tableau.  For $m$-row shapes $\lambda$ and standard content $\mu$, the total number of $k$-inverted standard Young tableau of shape $\lambda$ are then enumerated by relating such tableaux to $m$-dimensional generalizations of Dyck paths and counting the numbers of ``returns to ground" in those paths.  In the rectangular specialization of $\lambda = n^m$ this yields a generating function that involves $m$-dimensional analogues of the famed Ballot numbers.  Our various results are then used to directly enumerate all $k$-inverted semistandard Young tableaux with arbitrary content and two-row shape $\lambda = a^1 b^1$, as well as all $k$-inverted standard Young tableaux with two-column shape $\lambda=2^n$.

\end{abstract}

\section{Introduction}
\label{sec: intro}

For any non-increasing strong partition $\lambda = (\lambda_1,\lambda_2,\hdots,\lambda_m)$ of the positive integer $N$, a Young diagram $Y$ of shape $\lambda$ is a left-justified array of $N$ boxes such that there are precisely $\lambda_i$ boxes in the $i^{th}$ row of $Y$.  A filling of a Young diagram $Y$ is an assignment of positive integers (possibly repeated) to the boxes of $Y$ such that no integer is skipped.  If a filling uses precisely $\mu_i$ copies of $i$ for each integer $1 \leq i \leq M$, where $M \leq N$ and $\mu_1 + \hdots \mu_M = N$, we say that the filling has content $\mu = (\mu_1,\mu_2,\hdots, \mu_M)$.  We will oftentimes use the shorthand notation $\mu = 1^{\mu_1} 2^{\mu_2} \hdots M^{\mu_M}$.

This paper will utilize several distinct types of fillings.  A filling is said to be row-standard if entries increase from left-to-right across each row, and to be column-standard (resp. column-semistandard) if entries increase (resp. weakly increase) from top-to-bottom down each column.  If a filling is both row-standard and column-standard, as well as if $\mu = 1^1 2^1 \hdots N^1$, the resulting array $T$ is said to be a \textbf{standard Young tableau}.  If a filling is merely row-standard and column-semistandard, the resulting array is said to be a \textbf{semistandard Young tableau}.  We denote the set of all standard Young tableaux of shape $\lambda$ by $S(\lambda)$, and the set of all semistandard Young tableaux of shape $\lambda$ and content $\mu$ by $S(\lambda,\mu)$.  For a comprehensive introduction to Young tableaux, see Fulton \cite{Fulton}.

Our primary focus are inversions of Young tableaux, as first introduced by Fresse \cite{Fresse} to calculate the Betti numbers of Springer fibers in type A.\footnote{An entirely distinct notion also referred to as ``tableau inversions" has been presented by Shynar \cite{Shynar}.}  Adopting the terminology of Beagley and the author \cite{BD, Drube}, let $\tau$ be a row-standard filling and let $i,j$ be a pair of entries from the same column of $\tau$.  Let $i_k$ denote the entry precisely $k$ boxes to the right of $i$ in $\tau$, and let $j_k$ denote the entry precisely $k$ boxes to the right of $j$ in $\tau$.  Then $(i,j)$ form an \textbf{inversion pair} of $\tau$ if $i < j$ and one of the following holds:

\begin{enumerate}
\item Either $i_1$ or $j_1$ doesn't exist, and $i$ appears below $j$.
\item $i_1 > j_1$.
\item $i_k = j_k$ for all $1 \leq k \leq n$, either $i_{n+1}$ or $j_{n+1}$ doesn't exist, and $i$ appears below $j$.
\item $i_k = j_k$ for all $1 \leq k \leq n$, and $i_{n+1} > j_{n+1}$. 
\end{enumerate}

Taken together, the four conditions above identity an instance where the column containing $i,j$ is not in the correct (non-decreasing) order relative to the first column on its right where an appropriate order may be discerned.  As such, inversion pairs quantify how far a row-standard tableau is from also being column semistandard.  Notice that the absolute vertical placement of $i$ and $j$ within $\tau$ doesn't necessarily determine whether $(i,j)$ constitutes an inversion pair; this is an entirely ``local" phenomenon that only concerns the ordering of a column relative to the ordering of more rightward columns.  See Figure \ref{fig: inverted tableaux examples} for examples of row-standard fillings along with their inversion pairs.

\begin{figure}[ht!]
\centering
\begin{ytableau}
4 & 7 & 8 & 9 \\
2 & 3 & 6 \\
1 & 5
\end{ytableau}
\hspace{.5in}
\begin{ytableau}
3 & 5 & 6 & 7 \\
1 & 5 & 8 & 9\\
2 & 4 & 8
\end{ytableau}
\caption{A row-standard filling (left) with inversion pairs $(1,2)$, $(5,7)$, $(6,8)$, and a row-standard filling (right) with inversion pairs $(1,2)$, $(1,3)$, $(4,5)$, $(4,5)$.}
\label{fig: inverted tableaux examples}
\end{figure}

Observe that if $\tau$ possesses repeated entries, then it is possible for a single inversion pair to appear multiple times within $\tau$.  The existence of repeated inversion pairs only requires that the entries involved in the repeated inversion pairs $(i,j)$ aren't the same instances of $i$ and $j$.  As such, we will always list the inversion pairs of $\tau$ with multiplicity.  If it ever becomes necessary to specify that an inversion pair $(i,j)$ involves entries from the $k^{th}$ column of $\tau$, we adopt the notation $(i,j)^k$.  Furthermore, if we need to specify which instances of a repeated entry are involved in an inversion pair, we will use alphabetic subscripts $i_a$, $i_b$, etc. that have been indexed from top-to-bottom in the $k^{th}$ column of $\tau$.

If the row-standard filling $\tau$ contains precisely $k$ inversion pairs, we write $\inv(\tau) = k$ and say that $\tau$ is a \textbf{k-inverted semistandard Young tableau}.  More precisely, if $\inv(\tau) = k$ and $\tau$ lacks repeated entries we say that $\tau$ is a \textbf{k-inverted standard Young tableau}.  We denote the set of all $k$-inverted standard Young tableaux of shape $\lambda$ by $S_k(\lambda)$, and the set of all inverted standard Young tableaux of shape $\lambda$ by $I(\lambda) = \bigcup_k S_k(\lambda)$.  Similarly, we denote the set of all $k$-inverted semistandard Young tableaux with given $\lambda$ and $\mu$ by $S_k(\lambda,\mu)$, and the set of all inverted semistandard Young tableaux of given $\lambda$ and $\mu$ by $I(\lambda,\mu) = \bigcup_k S_k(\lambda,\mu)$.  Observe that the traditional notions of standard and semistandard Young tableaux correspond to $S_0(\lambda) = S(\lambda)$ and $S_0(\lambda,\mu) = S(\lambda,\mu)$, as a row-standard filling is also column-(semi)standard if and only if it has exactly zero inversion pairs.

Fresse \cite{Fresse} argued that the columns of any inverted standard Young tableau $\tau$ may be independently reordered to produce a unique column-standard Young tableau that we refer to as the \textbf{standardization} $\st(\tau)$ of $\tau$.  Fresse's work \cite{Fresse} also allows us to conclude that every inverted standard Young tableau is uniquely determined by its standardization and its collection of inversion pairs.  Fixing a standardization $T \in S(\lambda)$, we define $S_k^T(\lambda) = \lbrace \tau \in S_k(\lambda) \ \vert \ \st(\tau) = T \rbrace $ and $I^T(\lambda) = \lbrace \tau \in I(\lambda) \ \vert \ \st(\tau) = T \rbrace$.  The author \cite{Drube} later argued that unique standardizations $\st(\tau)$ also exist for inverted semistandard Young tableau.  Fixing semistandard $T \in S(\lambda,\mu)$, we adopt analogous notations $S_k^T(\lambda,\mu) = \lbrace \tau \in S_k(\lambda,\mu) \ \vert \ \st(\tau) = T \rbrace $ and $I^T(\lambda,\mu) = \lbrace \tau \in I(\lambda,\mu) \ \vert \ \st(\tau) = T \rbrace$.

One final piece of terminology that we will repeatedly use is the \textbf{height order} on entries in an inverted Young tableau.  As originally introduced by the author \cite{Drube}, you may define a complete order $\blacktriangleleft$ on the entries $\lbrace a_i \rbrace$ in each column of an inverted Young tableau $\tau$ by working through $\tau$ one column at a time, from right to left, according to the following rules:

\begin{itemize}
\item If either $a_i$ or $a_j$ lacks a entry directly to its right and $a_i$ appears above $a_j$, then $a_i \blacktriangleleft a_j$.
\item If $b_i$ is directly to the right of $a_i$, $b_j$ is directly right of $a_j$, and $b_i < b_j$, then $a_i \blacktriangleleft a_j$.
\item If $b_i$ is directly to the right of $a_i$, $b_j$ is directly right of $a_j$, and $b_i = b_j$ with $b_i \blacktriangleleft b_j$, then $a_i \blacktriangleleft a_j$.
\end{itemize}

If $c \in \tau$ is the $k^{th}$ smallest element in its column relative to the height order, we say that $c$ has \textbf{height} $k$ and write $\hgt(c)=k$.  The height order is constructed to tell us how the entries in a column need to be ordered (relative to the columns on its right) if one wants to completely avoid inversion pairs in that column.  This observation allows us to succinctly recast the definition of inversion pair:

\begin{proposition}[\textbf{Proposition 5 of \cite{Drube}}]
\label{thm: inversions from height order}
Let $\tau$ be an inverted Young tableau and let $i,j$ be distinct entries from the same column of $\tau$.  Then $(i,j)$ is an inversion pair of $\tau$ if and only if $i<j$ and $j \blacktriangleleft i$.
\end{proposition}

\subsection{Outline of Results}
\label{subsec: outline of results}

This paper is composed of three distinct yet interrelated pieces.  In Section \ref{sec: generating functions} we present a series of foundational results that formalize the theory of inverted semistandard Young tableau beyond what was already presented by Fresse \cite{Fresse} and the author \cite{Drube}.  This includes propositions describing what information is necessary to uniquely identify an inverted semistandard tableau from its standardization (Proposition \ref{thm: inversions uniquely identify}), and how any inverted semistandard tableau may be obtained from its standardization via a finite series of ``partial row transpositions" (Proposition \ref{thm: partial row transpositions generate}).  Section \ref{sec: generating functions} culminates with Theorem \ref{thm: fixed standardization generating function}, which gives a generating function $\chi^T(q)$ for the numbers $\vert S_k^T (\lambda,\mu) \vert$ of $k$-inverted semistandard Young tableaux with fixed standardization $T \in S(\lambda,\mu)$.  The $\dep$ and $\depm$ operators below represent positive integers known as ``inversions depths" that are easily read from $T$.

\begin{theorem}[Theorem \ref{thm: fixed standardization generating function}]
Let $T \in S(\lambda,\mu)$ be a semistandard Young tableaux.  Then:
$$\chi^T(q) \ = \ \displaystyle{\sum_k \vert S_k^T(\lambda,\mu) \vert q^k \ = \ \frac{\prod_{i,j} [\depm(a_{ij})+1]_q}{\prod_{i,j} [\depm(a_{ij}) - \dep(a_{ij}) + 1]_q}}$$
\end{theorem}

Section \ref{sec: generalized ballot numbers} explores the relationship between inverted standard Young tableau and various properties of higher-dimensional Dyck paths, expanding upon the well-known bijection between standard Young tableau $S(\lambda)$ of shape $\lambda$ and Dyck paths $\mathcal{D}_\lambda$ of the same shape.  In the two-row rectangular case, this allows us to give a non-$T$-specific generating function $\xi(q)$ for the $\vert S_k(\lambda) \vert$ in terms of the celebrated Ballot numbers $B(a,b)$, a result that eventually reappears as Corollary \ref{thm: two-row rectangular generating function as ballot numbers}:

\begin{corollary}[Corollary \ref{thm: two-row rectangular generating function as ballot numbers}]
Let $\lambda = (n,n)$.  Then the $S_k(\lambda)$ have generating function:
$$\xi(q) \ = \ \sum_k \vert S_k(\lambda) \vert q^k \ = \ \sum_i B(n-1,n-i) (1+q)^i \ = \ \sum_i \frac{i}{n} \binom{2n - i - 1}{n-i} (1+q)^i$$
\end{corollary}

That result is generalized to arbitrary $m$-row tableaux shapes in Theorem \ref{thm: m-row generating function as returns to ground} by relating our generating function $\xi(q)$ to ``generalized ballot numbers" $\vert \mathcal{D}_\lambda(k_1,\hdots,k_{m-1}) \vert$, which represent the number of Dyck paths in $\mathcal{D}_\lambda$ possessing various numbers of ``higher-dimensional returns-to-ground":

\begin{theorem}[Theorem \ref{thm: m-row generating function as returns to ground}]
Let $\lambda = (\lambda_1,\hdots,\lambda_m)$.  Then the $S_k(\lambda)$ have generating function:
$$\xi(q) \ = \ \sum_k \vert S_k(\lambda) \vert q^k \ = \ \sum_{(i_1,\hdots,i_{m-1})} \left( \vert \mathcal{D}_\lambda (i_1,\hdots,i_{m-1}) \vert \prod_{j=1}^{m-1} [j+1]_q^{i_j} \right)$$
\end{theorem}
 
In Section \ref{sec: direct enumeration} we utilize the results of earlier sections to directly enumerate $k$-inverted tableau in several cases that have not previously been attempted.  Theorem \ref{thm: two-row enumeration} explicitly enumerates $k$-inverted semistandard Young tableaux for an arbitrary two-row shape $\lambda=a^1 b^1$, generalizing earlier enumerations \cite{Fresse,Drube} of two-row inverted tableaux in the standard specialization.  Theorem \ref{thm: two-column generating function rectangular} then provides a generating function for the $\vert S_k(\lambda) \vert$ in the case of two-column rectangular shapes $\lambda = 2^n$.

As a whole, notice that this paper follows the purely combinatorial approach of Beagley and the author \cite{BD,Drube}, and will only briefly mention several potential geometric implications for Springer varieties (and Spaltenstein varieties) that may follow from the original work of Fresse \cite{Fresse}.

\section{Generating Functions for Inverted Semistandard Young Tableaux}
\label{sec: generating functions}

Before proceeding to our central enumerative results, we pause to lay the necessary combinatorial groundwork.  We begin this foundational section with a thorough treatment of ``formal" results that were not previously addressed by the author \cite{Drube}.  This formalism accomplishes for inverted semistandard tableaux what Fresse has already established for inverted standard tableaux \cite{Fresse}.  More importantly, it provides us with the vocabulary needed to precisely define an explicit generating function $\chi^T(q) = \sum_k \vert S_k^T(\lambda,\mu)\vert$ for the number of $k$-inverted semistandard Young tableaux with a fixed standardization $T \in S(\lambda,\mu)$.  This section closes with a series of corollaries demonstrating the usefulness of our generating function.

It should be noted that, although generating functions for the $\vert S_k^T(\lambda,\mu) \vert$ have yet to appear anywhere in the literature, the specialization of our $\chi^T(q)$ to the case of standard Young tableaux recovers Fresse's generating function for the $\vert S_k^T(\lambda) \vert$ (see Propositions 2.3b and 4.2 of \cite{Fresse}).  Yet even in the case of non-repeated entries, our terminology is more direct (avoiding the reference of appropriately defined sub-tableau) and thus allows for more quickly calculable generating functions.

As another brief aside, for the rest of this section we will often need to specify which instance of a repeated entry we are dealing with in a tableau.  We reserve the term ``entry" if we want to refer to a (potentially repeated) integer in a specific cell, and will use the notation $a_{ij}$ if we want to emphasize that the entry is located in the $i^{th}$ row and $j^{th}$ column of our inverted tableau.  If $a_{ij}$ is an instance of the integer $k$, we will call it an ``entry of value $k$" and write $a_{ij} = k$.

Our first foundational result follows most directly from the work of Fresse \cite{Fresse}.  It should be noted that the subsequent proposition may be independently derived via a careful reworking of the ``admissible transposition" procedure that appears in Proposition \ref{thm: partial row transpositions generate}.

\begin{proposition}
\label{thm: inversions uniquely identify}
Take any inverted semistandard Young tableau $\tau \in I(\lambda,\mu)$.  $\tau$ is uniquely identified by its standardization $st(\tau) = T$ alongside a collection of inversion pairs $(a_{ij},b_{ij})$ that specify the entries of $T$ involved in each inversion.
\end{proposition}
\begin{proof}
Take any inverted tableau $\tau \in I^T(\lambda,\mu)$ with standardization $\st(\tau) = T$, and assume $\sum_k \mu_k = N$.  For each value $\alpha$, we simultaneously place a complete order $\prec$ on the $\mu_\alpha$ copies of $\alpha$ in both $T$ and $\tau$.  Let $\prec$ be the unique complete order such that $\alpha_k \prec \alpha_l$ if $\alpha_k$ appears in a more leftward column than $\alpha_l$, and $\alpha_k \prec \alpha_l$ if $\alpha_k,\alpha_l$ appear in the same column and $\alpha_k \blacktriangleleft \alpha_l$.  Label the copies of $\alpha$ in both $T$ and $\tau$ according to this complete order, so that $\alpha_1 \prec \hdots \prec \alpha_{\mu_\alpha}$.  Then re-index all $N$ entries in both $T$ and $\tau$ according to the map $\phi(\alpha_k) = \mu_1 + \hdots + \mu_{(\alpha-1)}+k$.

This re-indexing replaces $T$ with a standard Young tableau $\widetilde{T} \in S(\lambda)$ and replaces $\tau$ with an inverted standard Young tableau $\widetilde{\tau} \in I^T(\lambda)$.  By construction, $\tau \mapsto \widetilde{\tau}$ represents a bijection between $I^T(\lambda,\mu)$ and all inverted tableau in $I^T(\lambda)$ that do not possess an inversion pair of the form $(a,b)$ with $\phi^{-1}(a) = \phi^{-1}(b)$.  In particular, $\inv(\tau) = \inv(\widetilde{\tau})$ and $(a,b)^j$ is an inversion pair of $\tau$ if and only if $(\phi(a),\phi(b))^j$ is an inversion pair of $\widetilde{\tau}$.  Citing Proposition 2.3a of Fresse \cite{Fresse}, we know that $\widetilde{\tau}$ is uniquely identified from $\widetilde{T}$ via its collection of inversion pairs.  If we specify which entries with a repeated value are involved in each inversion pair, the aforementioned bijection allows us to conclude that $\tau$ is uniquely identified from $T$ via a collection of entry-specific inversion pairs.
\end{proof}

Recall that, with standard tableaux, uniquely identifying an inverted tableaux required only a standardization and a list of the values involved in each inversion pairs.  In the semistandard case, as repeated inversion pairs $(a,b)$ are possible it was obviously necessary to be column-specific and to count inversion pairs with multiplicity.  Less obvious is the requirement that we must identify the exact location of entries involved in repeated inversion pairs.  See Figure \ref{fig: entries needed to specify inversions} for an example showing that a specific standardization as well as a collection of inversion pairs, even with column-specific information and accounting for multiplicity, is not enough to specify a unique inverted tableau.

\begin{figure}[ht!]
\centering
\scalebox{.85}{
\begin{ytableau}
2\\
1\\
1\\
2
\end{ytableau}
\hspace{.5in}
\begin{ytableau}
1\\
2\\
2\\
1
\end{ytableau}}
\caption{A pair of distinct tableaux in the same set $I(\lambda,\mu)$ with identical standardizations $T$ and identical inversion pairs $(1,2)^1,(1,2)^1$.  The two tableau are uniquely identified via the entry-specific inversion pairs sets $(1_a,2_a)^1,(1_b,2_a)^1$ and $(1_b,2_a)^1,(1_b,2_b)^1$}
\label{fig: entries needed to specify inversions}
\end{figure}

We now define an operation on inverted tableaux whose repeated application may be used to obtain any inverted tableau from its standardization.  So take $\tau \in I^T(\lambda,\mu)$, and let $a_{i_1j},a_{i_2j}$ be entries from the $j^{th}$ column of $\tau$ such that $a_{i_1j} < a_{i_2j}$ and $\vert \hgt(a_{i_1j}) - \hgt(a_{i_2j}) \vert = 1$.  A \textbf{partial row transposition} at $(a_{i_1j},a_{i_2j})$, denoted $a_{i_1j} \leftrightarrow a_{i_2j}$, is a transposition of the $i_1^{th}$ and $i_2^{th}$ rows of $\tau$ from the $j^{th}$ column leftward.  Now let $b_1$ denote the entry directly to the right of $a_1$ and let $b_2$ denote the entry directly to the right of $a_2$, if those entries in fact exist.  If $a_{i_1j}$ and $a_{i_2j}$ are both smaller than $b_1$ and $b_2$ (if they exist), the partial row transposition $a_{i_1j} \leftrightarrow a_{i_2j}$ preserves row-standardness and we refer to the operation as an \textbf{admissible (partial row) transposition}.

If $\tau \in I^T(\lambda,\mu)$, the fact that an admissible transposition $a_{i_1j} \leftrightarrow a_{i_2j}$ doesn't permute entries between columns means that the resulting tableau $\tau'$ is also an element of $I^T(\lambda,\mu)$.  $\tau'$ is related to $\tau$ in that it flips $\hgt(a_{i_1j})$ and $\hgt(a_{i_2j})$ while fixing the height of every other element.  As $a_{i_1j}$ and $a_{i_2j}$ are consecutive elements in the height order on the $j^{th}$ column of $\tau$, it follows that the admissible transposition $a_{i_1j} \leftrightarrow a_{i_2j}$ either adds or removes the inversion pair $(a_{i_1j},a_{i_2j})$ while leaving other inversion pairs unchanged.  These observations allow us to assert the following:

\begin{proposition}
\label{thm: partial row transpositions generate}
Every inverted tableau $\tau \in I^T(\lambda,\mu)$ may be obtained from its standardization $\st(\tau) = T$ via a finite sequence of admissible transpositions.  Furthermore, the minimum number of admissible transpositions needed to obtain $\tau$ from $T$ is $\inv(\tau)$.
\end{proposition}
\begin{proof}
Taking $\tau \in I^T(\lambda,\mu)$, we utilize Proposition \ref{thm: inversions uniquely identify} to uniquely identify $\tau$ via its collection of entry-specific inversion pairs $S = \lbrace (a_\alpha,b_\alpha)^{j_\alpha} \rbrace$.  As each admissible transposition adds precisely one inversion pair, obtaining $\tau$ from $T$ with fewer than $\inv(\tau)$ transpositions is clearly impossible.  To obtain $\tau$ with precisely $\inv(\tau)$ transpositions, we add inversion pairs one column at a time, from right to left.  Within each column, we work through entries from top-to-bottom via their placement in $T$, assuming that the relative ordering of repeated instances with a fixed value in the $j^{th}$ column is unchanged as we pass from $T$ through various elements of $I^T(\lambda,\mu)$.

So assume that we have followed this procedure up to the $j^{th}$ column of $T$, whose entries we denote $a_{1j} \leq a_{2j} \leq \hdots a_{hj}$ from top-to-bottom, and that we are ready to add inversions whose larger element is the entry $a_{ij}$.  Let $\widetilde{\tau} \in I^T(\lambda,\mu)$ denote the intermediate tableau that immediately precedes the addition of these inversion pairs.  Notice that, as $a_{ij}$ has yet to be the site of a partial row transposition, $a_{ij}$ still possesses its original height of $\hgt(a_{ij}) = i$ in $\widetilde{\tau}$.  Furthermore, the entries of lower height in the $j^{th}$ column of $\widetilde{\tau}$ are $a_{1j},\hdots,a_{(i-1)j}$, in some (possibly permuted) order.  Reverse index those entries according to their height order in $\widetilde{\tau}$ as $c_{i-1} \blacktriangleleft \hdots \blacktriangleleft c_1$.  If $a_{ij}$ is the larger entry in precisely $m$ inversion pairs in $\tau$, notice that those inversion pairs must be $(c_1,a_{ij})^j,\hdots,(c_m,a_{ij})^j$.  Otherwise, $\tau$ would have possessed an additional inversion pair of the form $(c_\beta,c_\gamma)$ that we had failed to account for at an earlier step.  Then perform the $m$ admissible transpositions $c_1 \leftrightarrow a_{ij}, \hdots, c_m \leftrightarrow a_{ij}$ in the stated order.  Continuing this procedure through all entries $a_{ij}$ of $T$ yields the desired inverted tableau $\tau$.
\end{proof}

See Figure \ref{fig: partial row transpositions} for an example of the procedure from the proof of Proposition \ref{thm: partial row transpositions generate}.  Notice that this procedure provides one minimal sequence of admissible transpositions that yields $\tau$ from $\st(\tau) = T$.  In general, there are many such sequences, even if one performs only the minimum possible number of $\inv(\tau)$ transpositions.  The author suspects that there exists a rich theory of composition patterns for admissible row transpositions of row-standard tableau.  As ordinary permutations $\sigma \in S_n$ correspond to inverted tableau of shape $\lambda = 1^n$, this topic would directly generalize the existing theory of reduced word decompositions for permutations.

\begin{figure}[ht!]
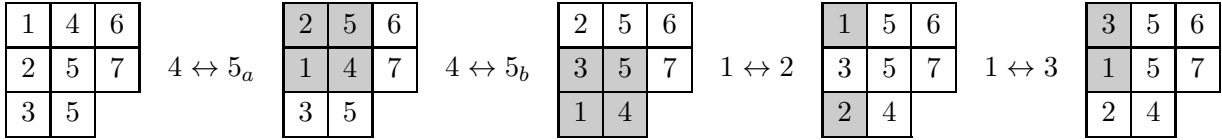

\centering
\begin{ytableau}
1 & 4 & 6 \\
2 & 5 & 7 \\
3 & 5
\end{ytableau}
\hspace{.05in}
\raisebox{-11pt}{$4 \leftrightarrow 5_a$}
\hspace{.05in}
\begin{ytableau}
*(gray!40) 2 & *(gray!40) 5 & 6 \\
*(gray!40) 1 & *(gray!40) 4 & 7 \\
3 & 5
\end{ytableau}
\hspace{.05in}
\raisebox{-11pt}{$4 \leftrightarrow 5_b$}
\hspace{.05in}
\begin{ytableau}
2 & 5 & 6 \\
*(gray!40) 3 & *(gray!40) 5 & 7 \\
*(gray!40) 1 & *(gray!40) 4
\end{ytableau}
\hspace{.05in}
\raisebox{-11pt}{$1 \leftrightarrow 2$}
\hspace{.05in}
\begin{ytableau}
*(gray!40) 1 & 5 & 6 \\
3 & 5 & 7 \\
*(gray!40) 2 & 4
\end{ytableau}
\hspace{.05in}
\raisebox{-11pt}{$1 \leftrightarrow 3$}
\hspace{.05in}
\begin{ytableau}
*(gray!40) 3 & 5 & 6 \\
*(gray!40) 1 & 5 & 7 \\
2 & 4
\end{ytableau}
\caption{Obtaining $\tau$ with inversion pairs $(1,2)^1,(1,3)^1,(4,5_a)^2,(4,5_b)^2$ from $\st(\tau)=T$ via a sequence of admissible transpositions.  $5_a$ and $5_b$ denote the two entries of value $5$ in second column, with $5_a \blacktriangleleft 5_b$.}
\label{fig: partial row transpositions}
\end{figure}

We are now set to develop the language needed for our generating function $\chi^T(q) = \sum_k \vert S_k^T(\lambda,\mu)\vert$.  So take a semistandard Young tableau $T \in S(\lambda,\mu)$, and let $a_{ij}$ be an entry of value $k$ from the $j^{th}$ column of $T$.  We define the \textbf{(standard) inversion depth} of $a_{ij}$, denoted $\dep(a_{ij})$, to equal the number of entries $b$ in the $j^{th}$ column of $T$ such that $b<k$ minus the number of entries $c$ in the $(j+1)^{st}$ column of $T$ such that $c<k$ (and this second number is taken to be zero if $T$ has no $(j+1)^{st}$ column).  We similarly define the \textbf{modified inversion depth} of $a_{ij}$, denoted $\depm(a_{ij})$ to equal the total number of entries above $a_{ij}$ in the $j^{th}$ column of $T$ minus the number of entries $c$ in the $(j+1)^{st}$ column of $T$ such that $c<k$ (once again taken to be zero if $T$ has no $(j+1)^{st}$ column).

For an example of a semistandard Young tableau alongside the inversion depths of each of its entries, standard and modified, see Figure \ref{fig: inversion depth}.  Notice that, if $a_{ij}$ is the only entry of value $k$ in the $j^{th}$ column of of $T$, then $\dep(a_{ij}) = \depm(a_{ij})$.  In particular, if $T$ is a standard Young tableaux, then $\dep(a_{ij}) = \depm(a_{ij})$ for all entries $a_{ij}$.  In the case of $T$ a standard Young tableaux, also notice that our $\dep(a_{ij})$ correspond with the $p_k$ defined by Fresse \cite{Fresse} as the number of appropriately long rows in the sub-tableau $T[1,\hdots,k]$.

\begin{figure}[ht!]
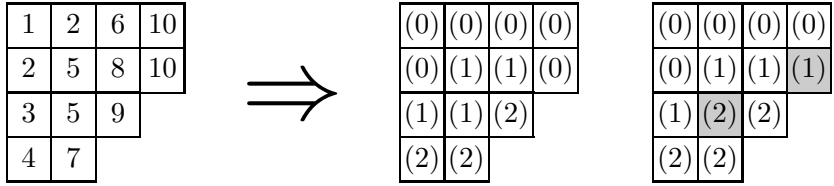

\centering
\begin{ytableau}
1 & 2 & 6 & 10\\
2 & 5 & 8 & 10\\
3 & 5 & 9\\
4 & 7
\end{ytableau}
\hspace{.2in}
\raisebox{-27pt}{\scalebox{3.5}{$\Rightarrow$}}
\hspace{.2in}
\begin{ytableau}
(0) & (0) & (0) & (0)\\
(0) & (1) & (1) & (0)\\
(1) & (1) & (2)\\
(2) & (2)
\end{ytableau}
\hspace{.3in}
\begin{ytableau}
(0) & (0) & (0) & (0)\\
(0) & (1) & (1) & *(gray!40)(1)\\
(1) & *(gray!40)(2) & (2)\\
(2) & (2)
\end{ytableau}
\caption{A semistandard Young tableaux with the inversion depth (center) and modified inversion depth (right) of each of its entries.  Entries where $\dep(a_{ij}) \neq \depm(a_{ij})$ have been highlighted.}
\label{fig: inversion depth}
\end{figure}

Standard inversion depth $\dep(a_{ij})$ has been defined to equal the number of distinct admissible partial row-transpositions possible in the $j^{th}$ column of $T$ where $a_{ij}$ is the larger entry being transposed.  This is equivalent to saying that $\tau \in I^T(\lambda,\mu)$ may possess at most $\dep(a_{ij})$ distinct inversion pairs of the form $(a_{i'j},a_{ij})^j$.

Despite placing a quick upper bound on $\inv(\tau)$ for $\tau \in I^T(\lambda,\mu)$, knowing the standard inversion depth of each element in $T$ will not be sufficient if one wishes to determine the generating function $\chi^T(q)$.  This is because standard inversion depth doesn't easily account for inversion pairs when the $j^{th}$ column of $T$ contains multiple entries of a fixed value $k$.  Luckily, modified inversion depth is perfectly suited to this task and immediately proves useful in the development of generating functions $\chi^T(q)$ for one-column semistandard Young tableaux.  As presented below, Lemma \ref{thm: generating function lemma} actually represents little more than a rewording of Theorem 13 from \cite{Drube}, which gave a generating function for the $\vert S_k(\lambda,\mu) \vert$ when $\lambda = 1^n$.

In what follows we use the standard notation for the $q$-number $[p]_q = 1 + q + \hdots + q^{p-1}$, the $q$-factorial $[p]_q! = [1]_q [2]_q \hdots [p]_q$, and the $q$-binomial coefficients $\binom{a}{b}_q = \frac{[a]_q!}{[b]_q![a-b]_q!}$.

\begin{proposition}
\label{thm: generating function lemma}
Let $T \in S(\lambda,\mu)$ be a one-column semistandard Young tableau of shape $\lambda = 1^n$.  Then:
$$\chi^T(q) = \displaystyle{\sum_k \vert S_k^T(\lambda,\mu) \vert q^k \ = \ \frac{\prod_i [\depm(a_{i1})+1]_q}{\prod_i [\mu_i]_q!}}$$
\end{proposition}
\begin{proof}
Notice that when $\lambda = 1^n$ that $S(\lambda,\mu)$ consists of a single column-semistandard tableau, which we label $T$.  Theorem 13 of \cite{Drube} showed that $\sum_k \vert S_k^T(\lambda,\mu) \vert q^k = \sum_k \vert S_k(\lambda,\mu) \vert q^k = \frac{[n]_q!}{\prod_i [\mu_i]_q!}$.  For a one-column tableau, the lack of rightward entries ensures $\depm(a_{i1}) = i-1$ for all $a_{i1}$.  Thus $[n]_q! = \prod_i [\depm(a_{i1}) + 1]_q$ and the result directly follows.
\end{proof}

Before generalizing Lemma \ref{thm: generating function lemma} to the multi-column case we pause to observe that, if $a_{ij}$ is the $m^{th}$ occurrence of value $k$ in the $j^{th}$ column of $T$, then $\dep(a_{ij}) + m - 1 = \depm(a_{ij})$.  It follows that $\depm(a_{ij}) - \dep(a_{ij})$ measures the number of instances of $k$ that lie above $a_{ij}$ in the $j^{th}$ column of $T$.

\begin{theorem}
\label{thm: fixed standardization generating function}
Let $T \in S(\lambda,\mu)$ be a semistandard Young tableaux.  Then:
$$\chi^T(q) \ = \ \displaystyle{\sum_k \vert S_k^T(\lambda,\mu) \vert q^k \ = \ \frac{\prod_{i,j} [\depm(a_{ij})+1]_q}{\prod_{i,j} [\depm(a_{ij}) - \dep(a_{ij}) + 1]_q}}$$
\end{theorem}
\begin{proof}
Proposition \ref{thm: inversions uniquely identify} has already shown that every inverted semistandard Young tableau is uniquely determined by its standardization and a collection of entry-specific inversion pairs.  Fixing $T \in S(\lambda,\mu)$, we merely need to determine which collections of inversion pairs are valid in the sense that they actually describe a tableau $\tau \in I^T(\lambda,\mu)$ that is row-standard.

When reordering the columns of $T$ to produce an arbitrary inverted tableau $\tau \in I^T(\lambda,\mu)$, notice that a rearrangement of entries in a leftward column in no way effects inversion pairs in rightward columns.  As inversion pairs are determined by height order (as opposed to vertical placement), we also know that a particular ordering of a rightward column does not effect whether a specific inversion pair is possible in a more leftward column.  These observations allow us to independently determine valid collections of inversion pairs for $I^T(\lambda,\mu)$ one column at a time.

So consider the $j^{th}$ column of $T$, and assume that the $j^{th}$ column contains precisely $c_k$ entries of value $k$.  Modifying the technique used in Theorem 13 of \cite{Drube}, we consider the number of different ways to ``build up" a valid height order on the $j^{th}$ column by recursively inserting all $c_k$ copies of $k$ into the height order of an ``intermediate" column with content $1^{c_1}2^{c_2} \hdots (k-1)^{c_{k-1}}$.  This yields a sequence $\lbrace \rho_1,\rho_2,\hdots,\rho_m \rbrace$ of partial columns such that $\rho_m$ describes the height order of the entire $j^{th}$ column of a row-standard tableau.

Notice that distinct insertions at any step in this procedure result in distinct height orders on the resulting column.  Also notice that a valid placement of the $c_k$ copies of $k$ is not subsequently made invalid upon the insertion of larger entries, as the placement of later entries may only increase the height order of a specific entry.  Lastly, note that the number of inversion pairs in the $j^{th}$ column whose larger element if $k$ is completely determined by the step in our procedure where we insert the $c_k$ copies of $k$ into $\rho_{k-1}$ to produce $\rho_k$.

So assume that we have a valid intermediate column $\rho_{k-1}$ and that we are prepared to insert the $c_k$ copies of $k$ in a way that preserves row-standardness.  By the definition of standard inversion depth, each instance $a_{ij}$ of $k$ may be involved in up to (a fixed number) of $\dep(a_{ij}) = \dep(k)$ inversion pairs where it is the larger element.  The number of such inversion pairs involving a specific instance $a_{ij}$ is determined by the number of entries in $\rho_{k-1}$ that it is moved ahead of in the height order.  If we want the resulting column $\rho_k$ to possess precisely $l$ inversion pairs whose larger entry is some instance $k$, it follows that such arrangements are in bijection with partitions of $l$ into at most $c_k$ parts where each part has size at most $\dep(k)$.

Recall that the coefficient of $q^l$ in $\binom{\alpha + \beta}{\alpha}_q$ equals the number of partitions of $l$ into at most $\alpha$ parts, each of which has size at most $\beta$.  If $k_1,\hdots,k_{c_k}$ denote the $c_k$ instances of $k$ in the $j^{th}$ column (read from top to bottom), this means that the number of different ways to produce a row-standard intermediate clumn $\rho_k$ with various numbers of inversions whose larger entry is $k$ has generating function:

$$f_{j,k}(q) \ = \ \binom{\dep(k) + c_k}{c_k}_q \ = \ \frac{[\dep(k)+1]_q \hdots [\dep(k)+c_k]_q}{[c_k]_q!} \ = \ \frac{\prod_{i=1}^{c_k}[\depm(k_i)+1]_q}{\prod_{i=1}^{c_k}[\depm(a_{ij}) - \dep(a_{ij}) + 1]_q}$$

The last equality above uses the aforementioned fact that the $m^{th}$ instance $k_m$ of value $k$ in the $j^{th}$ column satisfies $\dep(k_m) + m - 1 = \depm(k_m)$.  Ranging over all values in the $j^{th}$ column of $T$, and then over all columns in $T$, gives the formula from the theorem.
\end{proof}

\begin{example}
\label{ex: generating function}
If $T$ is the tableau of Figure \ref{fig: inversion depth}, Theorem \ref{thm: fixed standardization generating function} gives generating function:
$$\chi^T(q) \ = \ \sum_k \vert S_k^T(\lambda,\mu) \vert q^k \ = \ \frac{(1+q)^4 (1+q+q^2)^4}{(1+q)^2}$$
$$= \ 1 + 6q + 19q^2 + 40q^3 + 61q^4 + 70q^5 + 61q^6 + 40q^7 + 19q^8 + 6q^9 + q^{10}$$
\end{example}

We close this section by drawing a number of quick corollaries from Theorem \ref{thm: fixed standardization generating function}.  First notice that, if $T$ lacks repeated entries, then $\depm(a_{ij}) = \dep(a_{ij})$ for all $a_{ij}$ and the expression of Theorem \ref{thm: fixed standardization generating function} may be reduced to a generating function that is equivalent to the one introduced by Fresse \cite{Fresse}:

\begin{corollary}
\label{thm: generating function standard corollary}
Let $T \in S(\lambda)$ be a standard Young tableaux.  Then:
$$ \chi^T(q) \ = \ \sum_k \vert S_k^T(\lambda) \vert q^k \ = \ \prod_{i,j} [\dep(a_{ij})+1]_q$$
\end{corollary}

Now notice that our generating function $\chi^T(q)$ is always monic, meaning that there is precisely one ``maximally inverted" tableau $\tau \in I^T(\lambda,\mu)$ for each choice of $T \in S(\lambda,\mu)$.  This maximally inverted tableau always has degree $m_T = \prod_{i,j} \depm(a_{ij})/\prod_{i,j} (\depm(a_{ij}) - \dep(a_{ij}))$.

\begin{corollary}
\label{thm: maximal inversion number fixed standardization}
Let $T \in S(\lambda,\mu)$ be any semistandard Young tableau, and define $m_T$ as above.  Then $\vert S_k^T(\lambda,\mu) \vert = 0$ for all $k > m_T$ and $\vert S_{m_T}^T(\lambda,\mu) \vert = 1$.
\end{corollary}

Compare Corollary \ref{thm: maximal inversion number fixed standardization} to Theorem 7 of \cite{Drube}, in which the author determined a maximal inversion number $M_{\lambda,\mu}$ when ranging over all standardizations $T \in S(\lambda,\mu)$ and showed that precisely one $\tau_{max} \in I(\lambda,\mu)$ actually realized that number of inversions.\footnote{In \cite{Drube}, it is shown that $M_{\lambda,\mu} = \sum_j T_{(h_j -1)} - \sum_i T_{(\mu_i - 1)}$, where $T_k$ is the triangle number and $h_j$ is the height of the $j^{th}$ column in any tableau of shape $\lambda$.  The unique element of $I(\lambda,\mu)$ obtaining $M_{\lambda,\mu}$ inversions was labelled $\tau_{max}$, and its standardization was denoted $\st(\tau_{max}) = T^*$.}  Applying Corollary \ref{thm: maximal inversion number fixed standardization} to the language of Theorem 7 from \cite{Drube} allows us to conclude that the standardization $T^* = \st(\tau_{max})$ is the unique element of $S(\lambda,\mu)$ whose generating function $\chi^T(q)$ obtains the maximal possible degree of $M_{\lambda,\mu}$.

In order to draw our final corollaries, notice that $\chi^T(q)$ is always a product of palindromic unimodal polynomials of the form $\binom{\alpha}{\beta}_q$.  As the product of two palindromic unimodal polynomial is itself palindromic and unimodal (see Proposition 1 of Stanley \cite{Stanley}), we have both of the following:

\begin{corollary}
\label{thm: generating function symmetry}
Let $T \in S(\lambda,\mu)$ be any semistandard Young tableaux.  Then $\chi^T(q)$ is a palindromic polynomial.  In particular, if $m_T = \deg(\chi^T(q))$ then $\displaystyle{\vert S_k^T(\lambda,\mu) \vert = \vert S_{m_T - k}^T(\lambda,\mu) \vert}$ for all $k$.
\end{corollary}

\begin{corollary}
\label{thm: generating function unimodal}
Let $T \in S(\lambda,\mu)$ be any semistandard Young tableaux.  Then $\chi^T(q)$ is unimodal.
\end{corollary}

Pause to observe that the notion of unimodality addressed in Corollary \ref{thm: generating function unimodal} is distinct from the unimodality proven by Fresse, Mansour and Melnikov \cite{FMM}.  Corollary \ref{thm: generating function unimodal} proves the unimodality of $\chi^T(q) = \sum_k \vert S_k^T(\lambda,\mu) \vert$ for fixed standardization $T$ yet for arbitrary $\lambda,\mu$.  Fresse, Mansour and Melnikov restrict themselves to standard tableau and prove the unimodality of $\xi(q) = \sum_T \chi^T(q) = \sum_k \vert S_k(\lambda) \vert$ for specific (relatively) easy choices of $\lambda$, such as two-row and ``hook-shaped" cases.

Although the unimodality of $\chi^T(q)$ is nearly immediate, log-convexity of $\chi^T(q)$ does not necessarily follow because the $q$-binomial coefficients that constitute $\chi^T(q)$ need not be log-concave.  For example, take the sole semistandard tableau $T$ of $S(\lambda,\mu)$ for $\lambda = 1^4$ and $\mu = 1^2 2^2$ (i.e.- the standardization of both tableaux from Figure \ref{fig: entries needed to specify inversions}).  This $T$ has $\chi^T(q) = \binom{4}{2}_q = 1 + q + 2q^2 + q^3 + q^4$, which is palindromic and unimodal but not log-concave.

If however we restrict our attention to standard Young tableau, from Corollary \ref{thm: generating function standard corollary} we see that $\chi^T(q)$ is a product of log-concave polynomials of the form $[\alpha]_q$ and hence is itself log-concave:

\begin{corollary}
\label{thm: standard generating function log-concave}
Let $T \in S(\lambda)$ be any standard Young tableaux.  Then $\chi^T(q)$ is log-concave.
\end{corollary}

\section{Inverted Young Tableaux \& Generalized Ballot Numbers}
\label{sec: generalized ballot numbers}

In Section \ref{sec: generating functions} we fixed a semistandard Young tableau $T \in S(\lambda,\mu)$ and developed a generating function for the numbers of $k$-inverted tableau $\vert S_k^T(\lambda,\mu) \vert$ whose standardization was $T$.  It is natural to ask whether our results for the $\vert S_k^T(\lambda,\mu) \vert$ aid in the enumeration of the entire sets $S_k(\lambda,\mu)$ or in the development of a non-$T$-specific generating function $\xi(q) = \sum_T \chi^T(q) = \sum_k \vert S_k(\lambda,\mu) \vert$?  As noted by Beagley and the author \cite{BD,Drube}, this widening of scope is extremely difficult because it requires specific knowledge of every $T \in S(\lambda,\mu)$.  In particular, a direct generalization would require knowledge of how many $T \in S(\lambda,\mu)$ possess a fixed generating function $\chi^T(q)$.

It is because of these difficulties that previous attempts at enumerating the $S_k(\lambda,\mu)$ have disregarded the $S_k^T(\lambda,\mu)$ and utilized more direct techniques.  Even in the standard case, such techniques have only been previously applied to calculate $S_k(\lambda)$ for specific easy choices of $k$ and $\lambda$.  See Fresse, Mansour and Melnikov \cite{FMM} for calculations of $\vert S_k (\lambda) \vert$ when $\lambda$ is a two-row or ``hook" shape, as well as Beagley and the author \cite{BD} for an independent verification of the two-row case along with the $k=1$ case for an arbitrary shape $\lambda$.

The situation for semistandard tableaux is even more daunting, as there does not even exist a generalized ``hook-length formula" to enumerate semistandard Young tableaux $\vert S(\lambda,\mu) \vert$ with fixed shape and content.  The author \cite{Drube} was still able to determine $\vert S_k(\lambda,\mu) \vert$ for arbitrary $\mu$ when $\lambda$ was a two-row or one-column shape, conveniently corresponding to the cases where calculating $\vert S(\lambda,\mu) \vert$ is relatively straightforward.  For an arbitrary shape $\lambda$, the author \cite{Drube} also placed $S_1(\lambda,\mu)$ in bijection with $\bigcup_{\lambda_i} S_0(\lambda_i,\mu)$ for a certain finite collection of related shapes $\lambda_i$.

For the remainder of this paper, we present two new approaches for enumerating the $S_k(\lambda,\mu)$.  In this section, we restrict ourselves to the standard case and utilize the bijection between $m$-row standard Young tableaux and $m$-dimensional Dyck paths to explicitly determine how many $T \in S(\lambda)$ have a fixed generating function $\chi^T(q)$.  This is done first in the familiar two-dimensional case, which motivates the general $m$-row case.

In all that follows, let $S(\lambda,\mu) |_\chi \subset S(\lambda,\mu)$ denote the set of semistandard Young tableau that have a fixed generating function $\chi^T(q)$, a notation that we predictably adapt to the standard case as $S(\lambda) |_\chi$.  If $\lambda$ and $\mu$ are understood, we significantly shorten notation for the cardinalities of these sets to $\phi(\chi) = \vert S(\lambda,\mu) |_\chi \vert$.  Clearly $\xi(q) = \sum_\chi \phi(\chi) \chi^T(q)$.

\subsection{2-Dimensional Dyck Paths and Two-Row Inverted Young Tableaux}
\label{subsec: 2-dimensional dyck paths}

We begin by recapping basic results about traditional (two-dimensional) Dyck paths and how they relate to standard Young tableaux.  Let $\lambda = (a,b)$, where $a,b$ are positive integers.  A \textbf{Dyck path of shape} $\mathbf{\lambda}$ is an integer lattice path from $(0,0)$ to $(a,b)$ utilizing only ``East" $E = (1,0)$ and ``North" $N=(0,1)$ steps such that every point $(x,y)$ along the path satisfies $y \leq x$.  We denote the set of all Dyck paths of shape $\lambda$ by $\mathcal{D}_\lambda = \mathcal{D}_{(a,b)}$.  A particular Dyck path $P \in \mathcal{D}_\lambda$ will often be specified by $P = \lbrace v_0,\hdots,v_{a+b} \rbrace$, where the $v_i$ are the integer lattice points along the path.  Thus $v_0 = (0,0)$, $v_{a+b} = (a,b)$, and $v_i - v_{i-1} = (1,0)$ or $(0,1)$ for all $1 \leq i \leq a+b$.

Obviously $\mathcal{D}_{(a,b)} = \emptyset$ if $a>b$.  If $a=b=n$, it is well-known that the cardinality of $\mathcal{D}_{(n,n)} = \mathcal{D}_n$ equals the $n^{th}$ Catalan number $C_n = \frac{1}{n+1} \binom{2n}{n}$.  More generally, there is a well-studied bijection between $\mathcal{D}_\lambda$ and standard Young tableaux of shape $\lambda$.  See Figure \ref{fig: tableau to lattice path} for an example of this straightforward bijection, in which $E$ steps of the Dyck path correspond to values in the first row of the associated tableau and $N$ steps of the Dyck path correspond to values in the second row of the tableau.

\begin{figure}[ht!]
\centering
\raisebox{38pt}{\begin{ytableau}
1 & 2 & 5 & 7 \\
3 & 4 & 6
\end{ytableau}}
\hspace{.2in}
\raisebox{32pt}{\scalebox{2.5}{$\Leftrightarrow$}}
\hspace{.2in}
\begin{tikzpicture}
	[scale=.7,auto=left,every node/.style={circle, fill=black,inner sep=1.25pt}]
	\draw[dotted] (0,0) to (4,0);
	\draw[dotted] (0,1) to (4,1);
	\draw[dotted] (0,2) to (4,2);
	\draw[dotted] (0,3) to (4,3);
	\draw[dotted] (0,0) to (0,3);
	\draw[dotted] (1,0) to (1,3);
	\draw[dotted] (2,0) to (2,3);
	\draw[dotted] (3,0) to (3,3);
	\draw[dotted] (4,0) to (4,3);
	\draw (0,0) to (3,3);
	\node (0*) at (0,0) {};
	\node (1*) at (1,0) {};
	\node (2*) at (2,0) {};
	\node (3*) at (2,1) {};
	\node (4*) at (2,2) {};
	\node (5*) at (3,2) {};
	\node (6*) at (3,3) {};
	\node (7*) at (4,3) {};
	\node[fill=none] (1l) at (.5,-.25) {\scalebox{.75}{1}};
	\node[fill=none] (2l) at (1.5,-.25) {\scalebox{.75}{2}};
	\node[fill=none] (3l) at (2.2,.5) {\scalebox{.75}{3}};
	\node[fill=none] (4l) at (2.2,1.5) {\scalebox{.75}{4}};
	\node[fill=none] (5l) at (2.6,1.75) {\scalebox{.75}{5}};
	\node[fill=none] (6l) at (3.2,2.5) {\scalebox{.75}{6}};
	\node[fill=none] (7l) at (3.5,3.25) {\scalebox{.75}{7}};
	\draw[thick] (0*) to (1*);
	\draw[thick] (1*) to (2*);
	\draw[thick] (2*) to (3*);
	\draw[thick] (3*) to (4*);
	\draw[thick] (4*) to (5*);
	\draw[thick] (5*) to (6*);
	\draw[thick] (6*) to (7*);	
	\end{tikzpicture}
\caption{A standard Young tableau of shape $\lambda = (4,3)$ and the corresponding Dyck path in $\mathcal{D}_\lambda$.}
\label{fig: tableau to lattice path}
\end{figure}
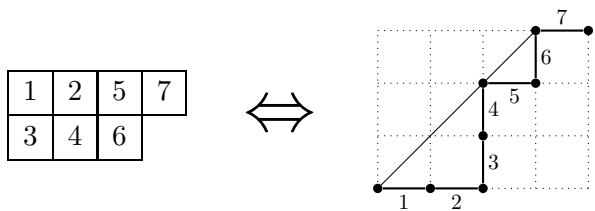

There are many important statistics on Dyck paths.  The statistic that will prove useful here is a Dyck path's number of ``returns to ground".  Formally, a path $P = \lbrace v_0,\hdots, v_{a+b} \rbrace$ in $\mathcal{D}_{(a,b)}$ has a \textbf{return to ground} at $v_i$ if $v_i = (x,x)$ for some positive integer $x$.  Observe that $v_0 = (0,0)$ does not qualify as a return to ground.  If $P \in \mathcal{D}_{\lambda}$ has precisely $k$ returns to returns to ground, we write $\ret(P)=k$.  We also establish the notation $\mathcal{D}_\lambda (k) = \lbrace P \in \mathcal{D}_\lambda \ \vert \ \ret(P) = k \rbrace$.

When $\lambda = (n,n)$, the sets $\mathcal{D}_n(k)$ are related to the famed ballot numbers $B(\alpha,\beta) = \frac{\alpha - \beta + 1}{\alpha + 1} \binom{\alpha + \beta}{\alpha}$ as $\vert \mathcal{D}_n(k) \vert = B(n-1,n-k) = \frac{k}{n} \binom{2n-k-1}{n-1}$.  These ballot numbers are frequently presented as entries in the so-called Catalan triangle of Figure \ref{fig: Catalan triangle}, whose rows sum to the Catalan numbers as $\sum_\beta B(\alpha,\beta) = C_\alpha$.  This means that our $\vert \mathcal{D}_n(k) \vert = B(n-1,n-k)$ may be read by proceeding from right-to-left along the appropriate row of Catalan's triangle.  See sequence A009766 on OEIS \cite{OEIS} for a thorough overview of Catalan triangle's, as well as Forder \cite{Forder} or Barcucci and Verri \cite{BV} for earlier investigations of the Ballot numbers.

\begin{figure}[ht!]
\centering
\begin{tabular}{c c c c c c}
1 & & & & & \\
1 & 1 & & & & \\
1 & 2 & 2 & & & \\
1 & 3 & 5 & 5 & & \\
1 & 4 & 9 & 14 & 14 & \\
1 & 5 & 14 & 28 & 42 & 42
\end{tabular}
\caption{Rows $0$ through $5$ of the Catalan triangle, whose $(\alpha,\beta)$ entry is the Ballot number $B(\alpha,\beta) = \frac{\alpha-\beta+1}{\alpha+1} \binom{\alpha+\beta}{\alpha}$.  Note that the leftmost column of the triangle corresponds to $\beta = 0$.}
\label{fig: Catalan triangle}
\end{figure}

Yet how does all of this relate to inverted standard Young tableaux?  In the example of Figure \ref{fig: tableau to lattice path}, notice that the associated tableau $T$ has $\dep(4)=\dep(6)=1$ and $\dep(i)=0$ for all other entries $i \in T$.  This suggests that a return to ground at $v_i$ in $P \in \mathcal{D}_\lambda$ is related to the inversion depth of $i$ in the associated standard Young tableau $T \in S(\lambda)$.  This proves to be a general phenomenon for any standard Young tableau of any two-row shape $\lambda$:

\begin{proposition}
\label{thm: returns to ground vs inversion depth}
Take $T \in S(\lambda)$ for some two-row shape $\lambda = (a,b)$, and let $P = \lbrace v_0,\hdots, v_{a+b} \rbrace$ be the associated Dyck path in $\mathcal{D}_\lambda$.  Then an entry $i$ of $T$ has $\dep(i) = 1$ if and only if $P$ has a return to ground at $v_i$.
\end{proposition}
\begin{proof}
Observe that an entry $i$ of $T$ has $\dep(i) = 1$ if and only if $i = 2m$ for some positive integer $m$, meaning that the first $m$ columns of $T$ coincide with some $\widetilde{T} \in S(\widetilde{\lambda})$ for $\widetilde{\lambda} = m^2$.  This occurs if and only if the first $2m$ steps in the associated $P \in \mathcal{D}_\lambda$ correspond to a Dyck path from $(0,0)$ to $(m,m)$.  Hence we may conclude that $\dep(i) = 1$ if and only if $v_i= (m,m)$ for some $m > 0$. 
\end{proof}

When $\lambda$ is a two-row shape, $\dep(i)=0$ and $\dep(i)=1$ are the only options for each entry $i$ of any $T \in S(\lambda)$.  Proposition \ref{thm: returns to ground vs inversion depth} then implies that an entry $i$ of $T$ has $\dep(i) = 1$ if the associated Dyck path $P \in \mathcal{D}_\lambda$ has a return to ground at $v_i$, while $\dep(i) = 0$ if $P$ does not have a return to ground at $v_i$.  This directly prompts our primary result of this subsection:

\begin{theorem}
\label{thm: two-row generating function as returns to ground}
Take any two-column shape $\lambda = (a,b)$.  Then the $S_k(\lambda)$ have generating function: 
$$\xi(q) \ = \ \sum_k \vert S_k(\lambda) \vert q^k \ = \ \sum_i \vert \mathcal{D}_\lambda(i) \vert (1+q)^i$$
\end{theorem}
\begin{proof}
Take any $T \in S(\lambda)$ and let $P \in \mathcal{D}_\lambda$ be the corresponding Dyck path.  By Corollary \ref{thm: generating function standard corollary} and Proposition \ref{thm: returns to ground vs inversion depth}, $\chi^T(q) = [2]_q^i = (1+q)^i$ if and only if $P \in \mathcal{D}_\lambda(i)$.  Ranging over all $T \in S(\lambda)$, for each $i \geq 1$ we have precisely $\vert D_\lambda(i) \vert$ contributions of the form $(1+q)^i$ to our overall generating function $\xi(q) = \sum_T \chi^T(q)$.  The stated result immediately follows.
\end{proof}

It should be noted that Fresse, Mansour and Melnikov \cite{FMM} have already derived a closed formula for each $\vert S_k(\lambda) \vert$ in the case of an arbitrary two-row shape $\lambda = (a,b)$.  Skip forward to Subsection \ref{subsec: two-row enumeration} to see our generalization of that formula to the semistandard case.  The reason that we still take time to derive the generating function of Theorem \ref{thm: two-row generating function as returns to ground}, apart from the fact that it highlights an interesting relationship between inverted tableau and Dyck paths, is that it easily generalizes to larger tableau for which closed formulas do not currently exist.  When $a=b=n$, Theorem \ref{thm: two-row generating function as returns to ground} also admits the following specialization that directly relates the $\vert S_k(\lambda) \vert$ to the Ballot numbers:

\begin{corollary}
\label{thm: two-row rectangular generating function as ballot numbers}
Let $\lambda = (n,n)$.  Then the $S_k(\lambda)$ have generating function:
$$\xi(q) \ = \ \sum_k \vert S_k(\lambda) \vert q^k \ = \ \sum_i B(n-1,n-i) (1+q)^i \ = \ \sum_i \frac{i}{n} \binom{2n - i - 1}{n-i} (1+q)^i$$
\end{corollary}

\begin{example}
\label{ex: two-row generating function} For $\lambda = (4,4)$ the $S_k(\lambda)$ have generating function:
$$\xi(q) \ = \ B(3,3)(1+q) + B(3,2)(1+q)^2 + B(3,1)(1+q)^3 + B(3,0)(1+q)^4$$
$$= \ 5(1+q) + 5(1+q)^2 + 3(1+q)^3 + 1(1+q)^4 \ = \ 14 + 28q + 20q^2 + 7q^3 + q^4$$
\end{example}

As one final comment about the two-row rectangular case, observe that the $\vert S_k(n,n) \vert$ resulting from Corollary \ref{thm: two-row rectangular generating function as ballot numbers} may be arranged into an integer triangle of their own.  This triangle is already known under a variety of different contexts as OEIS sequence A039599 \cite{OEIS}.  The relationship between the $\vert S_k(n,n) \vert$ and the other combinatorial interpretations of A039599 is an interesting topic for future investigation.

\subsection{m-Dimensional Dyck Paths and Inverted Young Tableaux of Shape $\lambda = n^m$}
\label{subsec: m-dimensional dyck paths}

Generalizing the methodology of Subsection \ref{subsec: two-row enumeration} to $m$-row standard Young tableaux requires $m$-dimensional analogues of Dyck paths.  So let $\lambda = (\lambda_1,\hdots,\lambda_m)$, where the $\lambda_i$ are positive integers, and set $N = \lambda_1 + \hdots \lambda_m$.  We adopt the standard notation of $\hat{e}_j$ as the unit vector in the $j^{th}$ coordinate of $\R^m$.  By an \textbf{m-dimensional Dyck path of shape} $\mathbf{\lambda}$ we mean a lattice path $P = \lbrace v_1,\hdots, v_N \rbrace$ with vertices $v_i \in \Z^m$ such that:

\begin{enumerate}
\item $v_0 = (0,\hdots,0)$ and $v_N = (\lambda_1,\hdots,\lambda_m)$.
\item For each $1 \leq i \leq N$ we have $v_i - v_{i-1} = \hat{e}_j$ for some $1 \leq j \leq m$.\\NOTE: If $v_i - v_{i-1} = \hat{e}_j$ we say that $P$ has a ``$j$-step" at $v_i$.
\item $v_i=(x_1,\hdots,x_m)$ satisfies $x_1 \leq x_2 \leq \hdots \leq x_m$ for all $1 \leq i \leq N$.
\end{enumerate}

By analogy with two-dimensional Dyck paths, we denote the set of all $m$-dimensional Dyck paths of shape $\lambda$ by $\mathcal{D}_\lambda = \mathcal{D}_{(\lambda_1,\hdots,\lambda_m)}$.  Also in direct analogy with two-dimensional Dyck paths, there is a well-known bijection between $\mathcal{D}_\lambda$ and $m$-row standard Young tableaux of shape $\lambda$.  Under this bijection, $P \in \mathcal{D}_\lambda$ has a $j$-step at $v_i$ if and only if $i$ appears in the $j^{th}$ row of the corresponding tableau $T$.  The fact that every point $v_i = (x_1,\hdots,x_m)$ along $P$ satisfies $x_1 \leq \hdots \leq x_m$ ensures that the corresponding tableau $T$ is in fact column standard.  Notice that, if $\lambda = n^m$ is a rectangular two-row shape, this bijection implies that $\vert \mathcal{D}_\lambda \vert$ is the so-called $m$-dimensional Catalan number $C_{d,n} = \left( \prod_{i=0}^{d-1} \frac{i!}{(n+i)!} \right) (dn)!$.  See OEIS sequence A060854 \cite{OEIS} or Gorska and Penson \cite{GP} for treatments of these numbers.

Before focusing upon the necessary properties of $m$-dimensional Dyck paths, we pause to establish a fact that follows quickly from the aforementioned bijection between $\mathcal{D}_\lambda$ and $S(\lambda)$.  Here we use the standard notation of $T[1,\hdots,i]$ to denote the sub-tableau of $T \in S(\lambda)$ that retains the cells of $T$ with values $1,\hdots,i$.

\begin{proposition}
\label{thm: Dyck path lattice points vs subtableau shape}
Take an $m$-row tableau $T \in S(\lambda)$ and let $P = \lbrace v_0,v_1,\hdots \rbrace$ be the corresponding $m$-dimensional Dyck path in $\mathcal{D}_\lambda$.  Then $T[1,\hdots,i]$ has shape $\lambda_i = (x_1,\hdots,x_m)$ if and only if $v_i = (x_1,\hdots,x_m)$.
\end{proposition}
\begin{proof}
$v_i = (x_1,\hdots,x_m)$ if and only if the first $i$ steps of $P$ contain $x_1$ $1$-steps, $x_2$ $2$-steps, etc.  This corresponds to the situation where precisely $x_1$ elements of $[i] = \lbrace 1,\hdots,i \rbrace$ have been assigned to the first row of $T$, precisely $x_2$ elements of $[i]$ have been assigned to the second row of $T$, etc.
\end{proof}

The relevant statistic on $m$-dimensional Dyck paths is a higher-dimensional analogue of our two-dimensional ``returns to ground".  These ``returns" will be points $v_i = (x_1,\hdots,x_m)$ along $P \in \mathcal{D}_\lambda$ at which we have some sort of equality within $x_1 \leq \hdots \leq x_m$.  The difficulty here is that, even in the three-dimensional case, different lengths of equalities are possible within $x_1 \leq \hdots \leq x_m$ and some of these lengths are possible with different collections of coordinates.

Take an $m$-dimensional Dyck path $P = \lbrace v_0,\hdots v_N \rbrace$ in $\mathcal{D}_\lambda$, and let $v_{i-1} = (x_1,\hdots,x_m)$, $v_i = (x_1',\hdots,x_m')$ be two consecutive lattice points along $P$.  By definition, $x_\gamma' = x_\gamma + 1$ for precisely one $1 \leq \gamma \leq m$ while $x_i' = x_i$ for all other $1 \leq i \leq m$.  We say that $P$ has a \textbf{d-degree return to ground (in the $\mathbf{\gamma}$ coordinate)} at $v_i$ if $x_{\gamma-d}' = x_{\gamma-d+1}' = \hdots = x_{\gamma}'$ yet $x_{\gamma-d-1} \neq x_{\gamma}$.  If $P$ has an $(m-1)$-degree return to ground in the $m^{th}$ coordinate at $v_i$, meaning that $v_i = (x,x,\hdots,x)$ for some positive integer $x$, we say that $P$ has a \textbf{full return to ground} at $v_i$.

As the situation described above ensures $x_{\gamma + 1}' > x_{\gamma}'$, the degree $d$ is one less than the length of the longest string of equalities in the coordinates of $v_i$ that did not already exist in the coordinates of $v_{i-1}$.  Clearly, $d$-degree returns to ground are only possible in the $\mathbf{\gamma}$ coordinate if $\gamma > d$.  In the case of a two-dimensional Dyck path, the only valid returns to ground are 1-degree returns to ground in the second coordinate, which qualify as full returns to ground.  In this sense, our definition is a natural extension of the pre-existing notion of returns to ground for two-dimensional Dyck paths.

If $P \in \mathcal{D}_\lambda$ has precisely $k_i$ $i$-degree returns to ground for each $i \geq 1$ (disregarding the coordinates in which those returns appear), we write $\Ret(P) = (k_1,k_2,\hdots) = \vec{k}$.  If $P \in \mathcal{D}_\lambda$ possesses precisely $k$ full returns to ground, we write $\ret(P) = k$ in direct analogy with the two-dimensional case.  Define $\mathcal{D}_\lambda(\vec{k}) = \lbrace P \in \mathcal{D}_\lambda \ \vert \ \Ret(P) = \vec{k} \rbrace$ and $\mathcal{D}_\lambda(k) = \lbrace P \in \mathcal{D}_\lambda \ \vert \ \ret(P) = k \rbrace$.

Even in the case of $\lambda = n^m$, there has been no previous attempt to enumerate the $\mathcal{D}_\lambda (\vec{k})$ or the $\mathcal{D}_\lambda (k)$.  Such an attempt is clearly outside the scope of this paper.  We do pause to note that, if $\lambda = n^m$, ranging over $n \geq 1$ allows one to compile both the $\vert \mathcal{D}_\lambda (\vec{k}) \vert$ and the $\vert \mathcal{D}_\lambda(k) \vert$ into analogues of the Catalan triangle.  For the $\vert \mathcal{D}_\lambda(\vec{k}) \vert$ this would take the form an $(m+1)$-dimensional array such that the set of all integers in each $m$-dimensional ``tier" sums to the $m$-dimensional Catalan number $C_{d,n}$.  For the $\vert \mathcal{D}_\lambda(k) \vert$ we have a two-dimensional array that presumably has much more in common with the original Catalan triangle, with each row in the array summing to $C_{d,n}$.  One could define the entries in this second array to be ``$m$-dimensional ballot numbers", namely $B^m(n-1,n-k) = \vert \mathcal{D}_\lambda(k) \vert$ when $\lambda = n^m$.

As in the two-dimensional case of Subsection \ref{subsec: 2-dimensional dyck paths}, the interest of this paper lies not in the inherent enumerative properties of the $ \mathcal{D}_\lambda (\vec{k})$ but in how the sizes of these sets determine generating functions for inverted standard Young tableaux.  The most general result, which holds for any $m$-row shape, is the following generalization of Proposition \ref{thm: returns to ground vs inversion depth}:

\begin{theorem}
\label{thm: d-degree returns to ground vs inversion depth}
For an $m$-row shape $\lambda = (\lambda_1,\hdots \lambda_m)$, take $T \in S(\lambda)$ and let $P = \lbrace v_0,v_1,\hdots \rbrace$ be the associated $m$-dimensional Dyck path in $\mathcal{D}_\lambda$.  Then the entry $a_{ij} = k$ of $T$ has $\dep(k) = d$ if and only if $P$ has a $d$-degree return to ground at $v_k$.
\end{theorem}
\begin{proof}
Take an entry $a_{ij} = k$ in the $i^{th}$ row and $j^{th}$ column of $T$, and let $v_k = (x_1,\hdots,x_m)$ be the associated lattice point in $P \in \mathcal{D}_\lambda$.  By definition, $\dep(k) = i - 1 - \beta_k$, where $\beta_k$ is the number of entries in the $(j+1)^{st}$ column of $T$ that are smaller than $k$.  By Proposition \ref{thm: Dyck path lattice points vs subtableau shape}, $\beta_k$ corresponds to the number of rows in $T[1,\hdots,k]$ whose length is strictly greater than $j$.  It follows that $\dep(k) = \gamma_k - 1$, where $\gamma_k$ equals the number of rows in $T[1,\hdots,k]$ of length precisely $j$.  Once again citing Proposition \ref{thm: Dyck path lattice points vs subtableau shape}, $\gamma_j$ equals the longest string of equalities in the coordinates of $v_k$ whose final coordinate is $x_i$.  By definition, we have that $P$ possesses a $(\gamma_k - 1)$-degree return to ground at $v_k$.  Both directions of the theorem statement immediately follow.
\end{proof}

Those familiar with the work of Fresse will notice that the notation of the preceding proof falls far closer to that of \cite{Fresse} than our own notation of Section \ref{sec: generating functions}.  Theorem \ref{thm: Dyck path lattice points vs subtableau shape} lends an unexpected convenience to his sub-tableau approach if one wants to invoke facts about the related Dyck paths, despite the fact that his work is entirely unconcerned with Dyck paths analogues.  Our usage of the subtableau $T[1,\hdots,k]$ actually doesn't extend beyond Theorem \ref{thm: Dyck path lattice points vs subtableau shape}, as the approach of Section \ref{sec: generating functions} once again becomes much more convenient in applying that result toward a generating function $\xi(q)$:

\begin{theorem}
\label{thm: m-row generating function as returns to ground}
Let $\lambda = (\lambda_1,\hdots,\lambda_m)$.  Then the $S_k(\lambda)$ have generating function:
$$\xi(q) \ = \ \sum_k \vert S_k(\lambda) \vert q^k \ = \ \sum_{(i_1,\hdots,i_{m-1})} \left( \vert \mathcal{D}_\lambda (i_1,\hdots,i_{m-1}) \vert \prod_{j=1}^{m-1} [j+1]_q^{i_j} \right)$$
\end{theorem}
\begin{proof}
Take $T \in S(\lambda)$ and let $P \in \mathcal{D}_\lambda$ be the associated Dyck path.  Corollary \ref{thm: generating function standard corollary} and Theorem \ref{thm: d-degree returns to ground vs inversion depth} together imply that $\chi^T(q) = \prod_{j=1}^{m-1} [j+1]_q^{i_j}$ if and only if $P \in \mathcal{D}_\lambda (i_1,\hdots,i_{m-1})$.  Ranging over all $T \in S(\lambda)$, there are precisely $\vert \mathcal{D}_\lambda (i_1,\hdots,i_{m-1}) \vert$ contributions of the form $\prod_{j=1}^{m-1} [j+1]_q^{i_j}$ to $\xi(q) = \sum_T \chi^T(q)$.  The theorem immediately follows.
\end{proof}

\section{Direct Enumerations of Inverted Young Tableaux}
\label{sec: direct enumeration}

In this final section, we present direct enumerations of the $\vert S_k(\lambda,\mu) \vert$ whose methodologies are completely distinct from the generalized Dyck path approach of Section \ref{sec: generalized ballot numbers}.  All of these methods make direct use of our generating function $\chi^T(q)$ from Theorem \ref{thm: fixed standardization generating function} or our notions of inversion depth.  This allows us to tackle a somewhat different collection of cases than previously attempted, and to generalize some pre-existing results to the semistandard case for the very first time.  In particular, we develop an explicit formula for the $\vert S_k(\lambda,\mu) \vert$ in the general two-row case of $\lambda = a^1 b^1$.  Opening up a wider range of tractable shapes by specializing to standard tableau, we close with an explicit enumeration of $\vert S_k(\lambda) \vert$ for the two-column rectangular case of $\lambda = 2^n$.

\subsection{Enumeration of $S_k(\lambda,\mu)$ for $\lambda = a^1 b^1$}
\label{subsec: two-row enumeration}

In the standard specialization, Fresse, Mansour and Melnikov \cite{FMM} have already explicitly enumerated the $S_k(\lambda)$ for a general two-row shape $\lambda = (a,b)$.  Theorem 2.1 of \cite{FMM} gives:

\begin{equation}
\label{eq: FMM two-row enumeration}
\vert S_k(\lambda) \vert \ = \ \beta_{b-k}^\lambda \ = \ \frac{a-b+1+2k}{a+1+k} \binom{a+b}{b-k}
\end{equation}

In this equation, $\beta_j^\lambda$ denotes the Betti number of the associated Springer variety $F_\lambda$, an algebraic interpretation that we do not make use of here.  The shift in subscript follows from the fact that $\dim(f_\lambda) = b$, a quantity that equals the maximal inversion number $M_\lambda$ for $\tau \in I(\lambda)$.

The $S_k(\lambda,\mu)$ were independently enumerated in the semistandard generalization by the author \cite{Drube}, but only in the rectangular case of $a=b=n$.  Theorem 14 of \cite{Drube} equated $\vert S_k(\lambda,\mu) \vert$ with a product of various Catalan numbers whose subscripts strongly partition $n$.  One can verify that this enumeration agrees with Equation \ref{eq: FMM two-row enumeration} as $\vert S_k(\lambda) \vert = \frac{1+2k}{n+1+k} \binom{2n}{n-k}$ for standard content $\mu$ if one rewrites the Catalan products of the author \cite{Drube} as an appropriate two-parameter Fuss-Catalan number (Raney number) via a modification of the results from Hilton and Pedersen \cite{HP}.

Here we directly extend Equation \ref{eq: FMM two-row enumeration} to the semistandard case in a way that recovers the formula of the author \cite{Drube} when $a=b=n$.  Essential to our technique is the invariance of $\vert S_k(\lambda,\mu) \vert$ under permutation of content, as originally presented by the author \cite{Drube}:

\begin{theorem}[\textbf{Theorem 11 of \cite{Drube}}]
\label{thm: permutation invariance}
For any shape $\lambda$, compatible content $\mu = 1^{\mu_1} 2^{\mu_2} \hdots M^{\mu_M}$, and permutation $\sigma \in S_M$, $\vert S_k(\lambda,\mu) \vert = \vert S_k (\lambda,\sigma(\mu)) \vert$ for all $k \geq 0$. 
\end{theorem}

For any two-row shape $\lambda = (a,b)$, clearly the only contents $\mu$ that allow for row-standard tableau are those where each value appears at most twice.  Using Theorem \ref{thm: permutation invariance}, we are justified in restricting our attention to $\mu$ where $1,2,\hdots,m$ each appear twice and $m+1,m+2,\hdots,a+b-2m$ each appear once.  This allows for the following:

\begin{theorem}
\label{thm: two-row enumeration}
Let $\lambda=(a,b)$ for any $a \geq b \geq 1$, and let $\mu = 1^{\mu_1} 2^{\mu_2} \hdots M^{\mu_M}$ be any content with $\sum_i \mu_i = a+b$.  If $\mu_i = 2$ for $m$ choices of $i$ and $\mu_i = 1$ for the remaining $a+b-2m$ choices of $i$, then:

$$\vert S_k(\lambda,\mu) \vert \ = \ \frac{a-b+1+2k}{a+1+k-m} \binom{a+b-2m}{b-k-m}$$
\end{theorem}
\begin{proof}
As previously argued, Theorem \ref{thm: permutation invariance} allows us to restrict our attention to the specific content $\mu$ where $1,2,\hdots,m$ each appear twice and $m+1,m+2,\hdots,a+b-2m$ each appear once.  Any tableau with this $\lambda$ and $\mu$ must take the form shown in Figure \ref{fig: two-row example}:

Such tableaux are clearly in bijection with standard Young tableaux of shape $\widetilde{\lambda} = (a-m,b-m)$, under the map that re-indexes entries in the final $b-m$ columns by $x \mapsto x-2m$.  For each $T \in S(\lambda,\mu)$, this bijection preserves the associated generating function $\chi^T(q)$.  This is due to the fact that all entries $a_{ij}$ in the first $m$ columns of $T$ have $\dep(a_{ij}) = \depm(a_{ij}) = 0$ and hence do not contribute to $\chi^T(q)$ by Theorem \ref{thm: fixed standardization generating function}.  It follows that $\xi(q) = \sum_k \vert S_k(\lambda,\mu) \vert = \sum_k \vert S_k(\widetilde{\lambda}) \vert$.  For any $k \geq 0$, directly applying Equation \ref{eq: FMM two-row enumeration} then gives:

$$\vert S_k (\lambda,\mu) \vert \ = \ \vert S_k (\widetilde{\lambda}) \vert \ = \ \frac{(a-m) - (b-m) + 1 + 2k}{(a-m)+1+k} \binom{(a-m) + (b-m)}{(b-m) - k}$$
\end{proof}

\begin{figure}[ht!]
\centering
\begin{ytableau}
1 & 2 & \hdots & m & *(gray!40) & *(gray!40) \hdots & *(gray!40) & *(gray!40) \hdots & *(gray!40) \\
1 & 2 & \hdots & m & *(gray!40) & *(gray!40) \hdots & *(gray!40)
\end{ytableau}

\caption{An arbitrary $T \in S(\lambda,\mu)$ with $\lambda = (a,b)$ and $\mu = 1^2 \hdots m^2 (m+1)^1 \hdots (a+b-2m)^1$.  Here the shaded cells correspond to some standard Young tableau under the re-indexing $x \mapsto x-2m$.}
\label{fig: two-row example}
\end{figure}

\subsection{Enumeration of $S_k(\lambda)$ for $\lambda = 2^n$}
\label{subsec: two-column enumeration}

Enumeration of the $\vert S_k(\lambda,\mu) \vert$ in the two-column rectangular case of $\lambda = 2^n$ has yet to appear anywhere in the literature, even for standard tableaux.  In this final subsection, we tackle the two-column case in the standard specialization.

Pause to observe that this entire subsection is made necessary because inversion depths, and hence the $\vert S_k(\lambda) \vert$, are not preserved under transposition.  If $\bar{\lambda}$ denotes the conjugate partition of $\lambda$ (obtained by transposing of the underlying Young diagram), it is well known that for standard Young tableaux we have $\vert S(\lambda) \vert = \vert S(\bar{\lambda}) \vert$.  However, it is typically not true that $\vert S_k(\lambda) \vert = \vert S_k(\bar{\lambda}) \vert$ for $k \geq 1$ if $\bar{\lambda}$ is in fact distinct from $\lambda$.  For a basic example of this fact, consider any one-row shape $\lambda = m^1$, so that $\bar{\lambda} = 1^m$.  Clearly $\vert S_k(\lambda) \vert = 0$ for all $k > 0$, yet $\vert S_k(\bar{\lambda}) \vert \neq 0$ for all $k \leq \binom{m}{2}$ via an application of Corollary \ref{thm: generating function standard corollary}.

To obtain a generating function $\xi(q)$ for the $\vert S_k (\lambda) \vert$, we determine which $T$-specific generating functions $\chi^T(q)$ are possible for $T \in S(\lambda)$, and then develop a closed formula for the number $\phi(\chi)$ of tableaux with a fixed $\chi^T(q)$.  As in Section \ref{sec: generalized ballot numbers}, this allows us to conclude that $\xi(q) = \sum_\chi \phi(\chi) \chi^T(q)$.  The first step in this process is a recognition of the fact that, for two-column tableaux, there is at most one $T \in S(\lambda)$ with a particular arrangement of inversion depths across its entries:

\begin{lemma}
\label{thm: two-column inversion depths uniquely identify}
Let $\lambda = 2^m 1^{n-m}$, and take $T \in S(\lambda)$.  If $a_1,\hdots,a_n$ are the entries in the first column of $T$, read from top to bottom, then $T$ is uniquely identified by the collection of inversion depths $\dep(a_i) = c_i$.
\end{lemma}
\begin{proof}
Pause to note that, if $b_1,\hdots,b_m$ are the entries in the second (rightmost) column of $T$, read from top to bottom, then $\dep(b_i) = i-1$ no matter our choice of $T$.  This is why we require only the inversion depths of the $a_i$ to distinguish tableaux in $S(\lambda)$.

So fix $T \in S(\lambda)$ and let $a_1,\hdots,a_n$ be the entries in the first column of $T$.  By definition, $\dep(a_i) = i - 1 - \beta_i$, where $\beta_i$ is the number of entries in the second column of $T$ that are smaller than $a_i$.  As $T$ possesses only two-columns, we also know that $a_i = i + \beta_i$ and hence that $\dep(a_i) = i - 1 - (a_i - i) = 2i - 1 - a_i$.  Given this direct relationship between the $a_i$ and the $\dep(a_i) = c_i$, it is clear that distinct tableaux in $S(\lambda)$ must possess distinct values of $c_i$ for at least one $i$.
\end{proof}

Using Lemma \ref{thm: two-column inversion depths uniquely identify} still requires us to characterize which sequences of inversion depths $c_1,c_2,\hdots$ are possible.  The necessary conditions on the $c_1,c_2,\hdots$ actually prove to be fairly straightforward:

\begin{lemma}
\label{thm: two-column valid inversion depths}
Set $\lambda = 2^n$, and let $c_1,\hdots,c_n$ be a sequence of non-negative integers.  Then there exists (unique) $T \in S(\lambda)$ whose first column entries $a_1,\hdots,a_n$ satisfy $\dep(a_i) = c_i$ for all $1 \leq i \leq n$ if and only if $c_1 = 0$ and $c_i \leq c_{i-1} + 1$ for all $i > 1$.
\end{lemma}
\begin{proof}
($\Rightarrow$) Take any $T \in S(\lambda)$.  Clearly, the top entry $a_1$ in the first column of $T$ must have an inversion depth of $0$, requiring $c_1 = 0$.  By the definition of inversion depth, for any $i>1$ there exist $\beta_{i-1} = (i-1)-1-c_{i-1}$ entries in the second column of $T$ that are smaller than $a_{i-1}$ and $\beta_i = i-1-c_i$ entries in the second column of $T$ smaller than $a_i$.  As $a_{i-1} < a_i$ we have $\beta_{i-1} \leq \beta_i$ and thus $c_i \leq c_{i-1} + 1$.

($\Leftarrow$) Now take a sequence of non-negative integers $c_1,\hdots,c_n$ satisfying the stated conditions.  We use that sequence to construct the tableau $T$ below, where $b_1,\hdots,b_n$ are the $n$ remaining integers not used in the first column (arranged in increasing order).

\begin{center}
\begin{tabular}{|c|c|}
\hline
$1 - c_1$ & $b_1$ \\ \hline
$3 - c_2$ & $b_2$ \\ \hline
$\vdots$ & $\vdots$ \\ \hline
$2n-1 - c_n$ & $b_n$ \\ \hline
\end{tabular}
\end{center}

Via equivalent reasoning to the proof of Lemma \ref{thm: two-column inversion depths uniquely identify}, if $a_i = 2i - 1 - c_i$ then $\dep(a_i) = 2i - 1 - (2i - 1 - c_i) = c_i$.  The assumptions that $c_1 = 0$ and $c_i \leq c_{i-1} + 1$ imply $c_i \leq i-1$ and hence $a_i \geq i$.  This ensures that the entries in the first column of $T$ are always positive integers between $1$ and $2n-1$.  By construction, $a_i \leq 2i-1$ and there are always enough larger entries leftover for the $b_i$ to ensure that $T$ is row-standard.  The fact that $c_i \leq c_{i-1} + 1$ also guarantees that $a_i \geq a_{i-1}$, allowing us to conclude that $T$ is column-standard.  Hence $T \in S(\lambda)$ is our standard Young tableau with the required inversion depths down its first column.
\end{proof}

Combining Lemmas \ref{thm: two-column inversion depths uniquely identify} and \ref{thm: two-column valid inversion depths}, we may conclude that there is precisely one $T \in S(\lambda)$ for every sequence of non-negative integers $c_1,\hdots,c_n$ such that $c_1 = 0$ and $c_i \leq c_{i-1} + 1$ for all $1 \leq i \leq n$.  One may use a generating tree to verify that the total number of such sequences is the Catalan number $C_n$, as expected from the hook-length formula applied to $\lambda = 2^n$.

Notice that if $T \in S(\lambda)$ is associated with the sequence $c_1,\hdots,c_n$, then Corollary \ref{thm: generating function standard corollary} gives generating function $\chi^T(q) = [n]_q! \prod_i [c_i + 1]_q$, where the leading $[n]_q!$ comes from the inversion depths of entries in the second column of $T$.  To identify the number $\phi(\chi)$ of $T \in S(\lambda)$ with a fixed generating function $\chi^T(q)$, we thus only need to determine how many valid sequences $c_1,\hdots,c_n$ have a particular number of zeroes, ones, twos, etc.

Once again pause to observe that the conditions on our sequences $c_1,\hdots,c_n$ guarantee that $c_i \leq i-1$ and hence that we only need to count occurrences of $0,1,\hdots,n-1$.  Henceforth let $\psi(\alpha_0,\hdots,\alpha_{n-1})$ denote the number of non-negative integer sequences $c_1,\hdots,c_n$ satisfying both $c_1 = 0$ and $c_i \leq c_{i-1} + 1$ that possess precisely $\alpha_k$ instances of $k$.  Clearly we must have $\sum_i \alpha_i = n$ in order for $\psi(\alpha_0,\hdots,\alpha_{n-1}) \neq 0$.  To achieve $\psi(\alpha_0,\hdots,\alpha_{n-1}) \neq 0$, the $c_i \leq c_{i-1} + 1$ condition also requires that a zero value $\alpha_k = 0$ never be followed by a nonzero value $\alpha_{k+1} \neq 0$.  One can show that these are the only two conditions needed to guarantee that $\psi(\alpha_0,\hdots,\alpha_{n-1}) \neq 0$, although we do not require this fact in the subsequent proof.

\begin{theorem}
\label{thm: two-column generating function rectangular}
Let $\lambda = 2^n$.  Then the $\vert S_k(\lambda) \vert$ have generating function:
$$\xi(q) \ = \ \sum_k \vert S_k(\lambda) \vert q^k \ = \ [n]_q! \left( \sum_{\alpha_0 + \hdots + \alpha_{n-1} = n} \left( \prod_{i=1}^{n-1} \binom{\alpha_i + \alpha_{i-1} - 1}{\alpha_i} [i+1]_q^{\alpha_i} \right) \right)$$
\end{theorem}
\begin{proof}
Our primary goal is to show that $\psi(\alpha_0,\hdots,\alpha_{n-1}) = \prod_{i=1}^{n-1} \binom{\alpha_i + \alpha_{i-1} - 1}{\alpha_i}$.  So assume that we want to construct an arbitrary sequence of non-negative integers $c_1,\hdots,c_n$ satisfying both $c_1=0$ and $c_i \leq c_{i-1} + 1$ that possesses precisely $\alpha_k$ instances of $k$ for each $k \geq 0$.  We recursively build up this sequence by simultaneously inserting all $\alpha_k$ copies of $k$ into a partial sequence that contains $\alpha_j$ copies of $j$ for all $0 \leq j \leq k-1$, making sure that our conditions on the $c_i$ are preserved after each step.

We begin by inserting the $\alpha_0$ copies of $0$ into an empty sequence, noticing that $c_1 = 0$ implies that we always have $\alpha_0 > 0$.  There is obviously only one way to perform this insertion.  Now assume that we have continued this process up to the insertion of the $\alpha_k$ copies of $k$.  Preserving the $c_i \leq c_{i-1} + 1$ condition requires that these instances of $k$ may only appear directly after an instance of $(k-1)$ or another instance of $k$.  This means that our previous placement of all instances of $0,1,\hdots,(k-2)$ is irrelevant to the validity of our placement at the step, and that we only need to consider the relative placement of the $\alpha_k + \alpha_{k-1}$ instances of $k$ and $(k-1)$.  It follows that there are precisely $\binom{\alpha_k + \alpha_{k-1} - 1}{\alpha_k}$ valid ways to insert our $\alpha_k$ copies of $k$.

Ranging over all $k$, we may conclude that there are precisely $\psi(\alpha_0,\hdots,\alpha_{n-1}) = \prod_{i=1}^{n-1} \binom{\alpha_i + \alpha_{i-1} - 1}{\alpha_i}$ valid sequences $c_1,\hdots,c_n$ that contain $\alpha_0$ copies of $0$, $\alpha_1$ copies of $1$, etc.  Notice that, if $\alpha_{k+1} \neq 0$ follows $\alpha_k = 0$, then $\binom{\alpha_{k+1} + \alpha_k - 1}{\alpha_{k+1}} = \binom{\alpha_{k+1} - 1}{\alpha_{k+1}} = 0$ and $\psi(\alpha_0,\hdots,\alpha_{n-1}) = 0$, as expected.

So assume that $c_1,\hdots, c_n$ possesses $\alpha_k$ copies of $k$ for all $0 \leq k \leq (n-1)$, and that $T \in S(\lambda)$ is any tableau associated to $c_1,\hdots,c_n$ via the bijection of Lemma \ref{thm: two-column valid inversion depths}.  Corollary \ref{thm: generating function standard corollary} then gives generating function $\chi^T(q) = [n]_q!\prod_{i=1}^{n-1} [i+1]_q^{\alpha_i}$.  The $\phi(\alpha_0,\hdots,\alpha_{n-1})$ tableaux with that generating function contribute a total of $\phi(\alpha_0,\hdots,\alpha_{n-1}) [n]_q!\prod_{i=1}^{n-1} [i+1]_q^{\alpha_i} = [n]_q! \prod_{i=1}^{n-1} \binom{\alpha_i + \alpha_{i-1} - 1}{\alpha_i} [i+1]_q^{\alpha_i}$ to the overall generating function $\xi(q) = \sum_k \vert S_k(\lambda) \vert q^k$.  Ranging over all combinations of the $\alpha_i$ such that $\alpha_0 + \hdots + \alpha_{n-1} = n$ then gives the desired result.
\end{proof}

\begin{example}
\label{ex: two-column example}
For $\lambda = 2^3$, we have $\dep(a_{ij}) \leq 2$ for all entries $a_{ij}$ of $T \in S(\lambda)$.  The only non-zero values for $\psi(\alpha_0,\alpha_1,\alpha_2)$ are shown below, followed by the generating function for the $\vert S_k(\lambda) \vert$:

$$\psi(3,0,0) = 1 \hspace{.3in} \psi(2,1,0) = 2 \hspace{.3in} \psi(1,2,0) = 1 \hspace{.3in} \psi(1,1,1) = 1$$
$$\xi(q) \ = \ [3]_q! \left( 1[1]_q^3 + 2[1]_q^2 [2]_q + 1 [1]_q [2]_q^2 + 1 [1]_q [2]_q [3]_q \right)$$
$$= \ 5 + 16x + 25x^2 + 24x^3 + 14x^4 + 5x^5 + x^6$$
\end{example}

A methodology similar to the one above may be applied to the non-rectangular two-column case of $\lambda = 2^m 1^{n-m}$, as tableaux of this shape are still uniquely identified by the inversion depths of entries in their first column.  Sadly, a generalization of the generating function from Theorem \ref{thm: two-column generating function rectangular} becomes so convoluted that we do not attempt a full derivation here.  The difficulty in generalizing to $\lambda = 2^m 1^{n-m}$ is that a different (and more complicated) collection of inversion depths $c_1,\hdots,c_n$ are now possible.  In particular, entries $a_i$ in the one-column ``tail" of $T \in S(\lambda)$ now have a nonzero lower bound on $\dep(a_i) = c_i$.  One may quickly verify that inversion depths in the first column of $T \in S(\lambda)$ fall within the ranges shown in Figure \ref{fig: two-column tableau via inversion depths}.

\begin{figure}[ht!]
\centering
\scalebox{.9}{
\begin{tabular}{|c|c|}
\cline{1-2}
$(0)$ & $(0)$ \\ \cline{1-2}
$0 \leq c_2 \leq 1$ & $(1)$ \\ \cline{1-2}
$\vdots$ & $\vdots$ \\ \cline{1-2}
$0 \leq c_m \leq m-1$ & $(m-1)$ \\ \cline{1-2}
$0 \leq c_{m+1} \leq m$ & \multicolumn{1}{c}{ } \\ \cline{1-1}
$1 \leq c_{m+2} \leq m+1$ & \multicolumn{1}{c}{ } \\ \cline{1-1}
$\vdots$ & \multicolumn{1}{c}{ } \\ \cline{1-1}
$n-m-1 \leq c_{m+n} \leq n-1$ & \multicolumn{1}{c}{ } \\ \cline{1-1}
\end{tabular}}
\caption{Possible inversion depths for entries in an arbitrary $T \in S(\lambda)$ of shape $\lambda = 2^m 1^{n-m}$.}
\label{fig: two-column tableau via inversion depths}
\end{figure}

With the ranges of Figure \ref{fig: two-column tableau via inversion depths} in mind, we close this paper with the following generalization of Lemma \ref{thm: two-column valid inversion depths}.  The following proposition would be the first step in developing a closed generating function $\xi(q)$ for the $\vert S^k(\lambda) \vert$ when $\lambda = 1^m 2^{n-m}$.

\begin{proposition}
\label{thm: two-column non-rectangular valid inversion depths}
Set $\lambda=2^m 1^{n-m}$, and let $c_1,\hdots,c_n$ be a sequence of non-negative integers.  Then there exists (unique) $T \in S(\lambda)$ whose first column entries $a_1,\hdots,a_n$ satisfy $\dep(a_i) = c_i$ for all $1 \leq i \leq n$ if and only if $c_1 =0$, $c_i \leq c_{i-1} + 1$ for all $i > 1$, and $c_i>i-m-1$ for all $i>m$.
\end{proposition}
\begin{proof}
The method of proof is directly equivalent to that of Lemma \ref{thm: two-column valid inversion depths}.  For the ($\Leftarrow$) direction, a valid sequence $c_1,\hdots,c_n$ of integers is now associated with the unique tableau shown below, which is a minor modification of the tableau from Lemma \ref{thm: two-column valid inversion depths}.

\begin{center}
\scalebox{.9}{
\begin{tabular}{|c|c|}
\hline
$1-c_1$ & $b_1$ \\ \cline{1-2}
$\vdots$ & $\vdots$ \\ \cline{1-2}
$2m-1 - c_m$ & $b_m$ \\ \cline{1-2}
$2m+1 - c_{m+1}$ & \multicolumn{1}{c}{ } \\ \cline{1-1}
$2m+3 - c_{m+2}$ & \multicolumn{1}{c}{ } \\ \cline{1-1}
$\vdots$ & \multicolumn{1}{c}{ } \\ \cline{1-1}
$2n - 1 - c_n$ & \multicolumn{1}{c}{ } \\ \cline{1-1}
\end{tabular}}
\end{center}
\end{proof}

\end{document}